\documentclass[11pt]{amsart}
\usepackage{amsmath,amssymb,latexsym,soul,cite,mathrsfs}

\usepackage{color,enumitem,graphicx}
\usepackage[colorlinks=true,urlcolor=blue,
citecolor=red,linkcolor=blue,linktocpage,pdfpagelabels,
bookmarksnumbered,bookmarksopen]{hyperref}
\usepackage[english]{babel}

\usepackage[left=2.9cm,right=2.9cm,top=2.8cm,bottom=2.8cm]{geometry}
\usepackage[hyperpageref]{backref}

\usepackage[colorinlistoftodos]{todonotes}
\makeatletter
\providecommand\@dotsep{5}
\def\listtodoname{List of Todos}
\def\listoftodos{\@starttoc{tdo}\listtodoname}
\makeatother

\numberwithin{equation}{section}

\def\R {{\rm I}\hskip -0.85mm{\rm R}}
\def\N {{\rm I}\hskip -0.85mm{\rm N}}

\newtheorem{theorem}{Theorem}[section]
\newtheorem{proposition}[theorem]{Proposition}
\newtheorem{lemma}[theorem]{Lemma}
\newtheorem{corollary}[theorem]{Corollary}
\newtheorem{definition}[theorem]{Definition}

\newtheorem{remark}{Remark}

\title[Existence and uniqueness of solution for a nonhomogeneous nonlocal problem]
{Existence and uniqueness of solution for a nonhomogeneous nonlocal problem}

\author[C. S. Z. Redwan]{Camil S. Z. Redwan}
\author[J. R. Santos Jr.]{Jo\~ao R. Santos J\'unior}
\author[A. Su\'arez]{Antonio Su\'arez}

\address[C. S. Z. Redwan]{\newline\indent Faculdade de Matem\'atica
\newline\indent 
Instituto de Ci\^{e}ncias Exatas e Naturais
\newline\indent 
Universidade Federal do Par\'a
\newline\indent
Avenida Augusto corr\^{e}a 01, 66075-110, Bel\'em, PA, Brazil}
\email{\href{mailto: camilredwan@gmail.com}{camilredwan@gmail.com}}

\address[J. R. Santos Jr.]{\newline\indent Faculdade de Matem\'atica
\newline\indent 
Instituto de Ci\^{e}ncias Exatas e Naturais
\newline\indent 
Universidade Federal do Par\'a
\newline\indent
Avenida Augusto corr\^{e}a 01, 66075-110, Bel\'em, PA, Brazil}
\email{\href{mailto: joaojunior@ufpa.br }{joaojunior@ufpa.br}}

\address[A. Su\'arez]{\newline\indent Departamento de Ecuaciones Diferenciales y An{\'a}lisis Num{\'e}rico
\newline\indent 
Facultad de Matem{\'a}ticas
\newline\indent 
Universidad de Sevilla
\newline\indent
C/. Tarfia s/n, 41012, Sevilla, Spain.}
\email{\href{mailto: suarez@us.es}{suarez@us.es}}

\thanks{Camil S. Z. Redwan was partially supported by Fapespa, Brazil. Jo\~ao R. Santos J\'unior was partially supported by CNPq-Proc. 302698/2015-9, Brazil. Antonio Su\'arez has been partially supported by MTM2015-69875-P (MINECO/FEDER, UE) and  CNPq-Proc. 400426/2013-7}

\begin{document}

\maketitle
\begin{abstract}
In this paper we investigate a class of elliptic problems involving a nonlocal Kirchhoff type operator with variable coefficients and data changing its sign. Under appropriated conditions on the coefficients, we have shown existence and uniqueness of solution.
\end{abstract}
\maketitle

\section{Introduction}


In this paper we are concerned with uniqueness of nontrivial classic solution to the following class of nonlocal elliptic equations

\begin{equation}\label{P}\tag{P}
\left \{ \begin{array}{ll}
-\left(a(x)+b(x)\int_{\Omega}|\nabla u|^{2}dx\right)\Delta u=h(x) & \mbox{in $\Omega$,}\\
u=0 & \mbox{on $\partial\Omega$,}
\end{array}\right.
\end{equation}
where $\Omega\subset\R^{N}$, $N\geq 2$, is a bounded domain with smooth boundary, $a, b\in C^{0, \gamma}(\overline{\Omega})$, $\gamma\in (0, 1)$, are positive functions with $a(x)\geq a_{0}>0$, $b(x)\geq b_{0}>0$ and $h\in  C^{0, \gamma}(\overline{\Omega})$ is given.

\medskip

When functions $a, b$ are positive constants, problem \eqref{P} is the $N$-dimensional stationary version of a hyperbolic problem proposed in \cite{K} to model small transversal vibrations of an elastic string with fixed ends which is composed by a homogeneous material. Such equation is a more realistic model than that provided by the classic D'Alembert's wave equation because takes account the changing in the length of the string during the vibrations. The hyperbolic Kirchhoff problem (with $a, b$ constants) began receiving special attention mainly after that in \cite{L} the author used an approach of functional analysis to attack it. 

At least in our knowledge, the first work in studying uniqueness questions to problem \eqref{P} with $a, b$ constants was \cite{ACM}. It is an immediate consequence of Theorem 1 in \cite{ACM} that if $h$ is a H\"older continuous nonnegative (nonzero) function then problem \eqref{P}, with $a, b$ constants, has a unique positive solution. In \cite{ACM} functions $h$ sign changing are not considered. In the case that $a, b$ are not constant, problem \eqref{P} is yet more relevant in an applications point of a view because its unidimensional version models small transversal vibrations of an elastic string composed by non-homogeneous materials (see \cite{FMSS}, section 2). In \cite{FMSS} (see Theorem 1) the authors proved that for each $h\in L^{\infty}(\Omega)$ ($h\not\equiv 0$) given, problem (\ref{P}) admits at least a nontrivial solution. Moreover, same article tells us that if $h$ has defined sign ($h\leq 0$ or $h\geq 0$) such a solution is unique. Unfortunately, since their approach was based in a monotonicity argument which does not work when $h$ is a sign changing function, they were not able to say anything about uniqueness in this case. Indeed, at least in our knowledge, actually the uniqueness of solution to problem \eqref{P} in the general case is an open problem. 

At this article we have obtained sufficient conditions on the quotient $a/b$ to ensure uniqueness of solution when function $h$, given, changes its sign. The main results of this paper are as follows

\begin{theorem}\label{Main1}
If there exists $\theta>0$ such that $a/b=\theta$ in $\Omega$, then for each $h\in C^{0, \gamma}(\overline{\Omega})$ problem \eqref{P} has a unique solution.
\end{theorem}

\begin{theorem}\label{Main2}
Let $a, b\in C^{2, \gamma}(\overline{\Omega})$, $h\in C^{0, \gamma}(\overline{\Omega})$ is sign changing and suppose $c=a/b$ not constant. 
\begin{itemize}
\item[$(i)$] If $\Delta c\geq 2|\nabla c|^{2}/c \ \mbox{in} \ \Omega$, then, for each $h\in C^{0, \gamma}(\overline{\Omega})$ given, problem \eqref{P} has a unique nontrivial classic solution.

\item[$(ii)$] If $\Delta c< 2|\nabla c|^{2}/c \ \mbox{in some open} \ \Omega_{0}\subset \Omega$ then, for each $h\in C^{0, \gamma}(\overline{\Omega})$ given, problem \eqref{P} has a unique nontrivial classic solution, provided that
$$
\frac{|\nabla c|_{\infty}c_{M}}{\sqrt{\lambda_{1}}c_{L}^{2}}\leq 3/2,
$$
where $\lambda_{1}$ is the first eigenvalue of laplacian operator with Dirichlet boundary condition, $c_{L}=\min_{x\in\overline{\Omega}}c(x)$, $c_{M}=\max_{x\in\overline{\Omega}}c(x)$ and $|\nabla c|_{\infty}=\max_{x\in\overline{\Omega}}|\nabla c (x)|$.
\end{itemize}
\end{theorem}

First theorem above generalizes Theorem 1 in \cite{ACM} because it is true to functions $h$ sign changing or not. The second theorem above complements Theorem 1 in \cite{FMSS}. 

\medskip

The paper is organized as follows. 

In Section \ref{se:prelim} we present some abstracts results, notations and definitions.
In Section \ref{se:eigen} we investigated a nonlocal eigenvalue problem which seems to be closely related with uniqueness questions to problem \eqref{P}.
In Section \ref{se:trivial} we prove Theorems \ref{Main1} and \ref{Main2}. Moreover, an alternative proof of the existence and uniqueness result in \cite{FMSS} is supplied.


\medskip

\section{Preliminaries}\label{se:prelim}

In this section we state some results and fix notations used along of paper.

\begin{definition}
We say that a function $h$ is signed in $\Omega$ if $h\geq 0$ in $\Omega$ or $h\leq 0$ in $\Omega$. 
\end{definition}

\begin{definition}
An application $\Psi:E\to F$ defined in Banach spaces is locally invertible in $u\in E$ if there are open sets $A\ni u$ in $E$ and $B\ni \Psi(u)$ in $F$ such that $\Psi:A\to B$ is a bijection. If $\Psi$ is locally invertible in any point $u\in E$ it is said that $\Psi:E\to F$ is locally invertible.
\end{definition}

\begin{definition}
Let $M, N$ be metric spaces. We say that a map $\Psi:M\to N$ is proper if $\Psi^{-1}(K)=\{u\in M: \Psi(u)\in K\}$ is compact in $M$ for all compact set $K\subset N$.
\end{definition}

Below we enunciate the classic local and global inverse function theorems, whose proofs can be found, for instance, in \cite{AP}.

\begin{theorem}[Local Inverse Theorem]\label{t1}
Let $E, F$ be two Banach spaces. Suppose $\Psi\in C^{1}(E, F)$ and $\Psi'(u):E\to F$ is a isomorphism. Then $\Psi$ is locally invertible at $u$ and its local inverse, $\Psi^{-1}$, is also a $C^{1}$-function.
\end{theorem}

\begin{theorem}[Global Inverse Theorem]\label{t2}
Let $M, N$ be two metric spaces and $\Psi\in C(M, N)$ a proper and locally invertible function on all of $M$. Suppose that $M$ is arcwise connected and $N$ is simply connected. Then $\Psi$ is a homeomorphism from $M$ onto $N$.  
\end{theorem}

Next, we state another classical result which will be used in our arguments and whose proof can be found, for instance, for a more general class of problems, in \cite{DG}.

\begin{proposition}\label{Prodi}
Let $m\in L^{\infty}(\Omega)$, $m(x)>0$ in a set of positive measure. Then, problem
\begin{equation}
\left \{ \begin{array}{ll}
-div(A(x)\nabla u) =\lambda m(x)u & \mbox{in $\Omega$,}\\
u=0 & \mbox{on $\partial\Omega$,}
\end{array}\right.
\end{equation}
where $A\in L^{\infty}(\Omega)$ and $A(x)\geq \mathfrak{m}$ for some positive constant $\mathfrak{m}$, has a smallest positive eigenvalue $\lambda_{1}(m)$ which is simple and corresponding eigenfunctions do not change sign in $\Omega$.
\end{proposition}

Throughout this paper $X$ is the Banach space 
$$
X=\{u\in C^{2, \gamma}(\overline{\Omega}): u=0 \ \mbox{on} \ \partial\Omega\}
$$ 
with norm
$$
\|u\|_{X}=\|u\|_{C^{2}(\overline{\Omega})}+\max_{|\beta|=2}[D^{\beta}u]_{\gamma},
$$ 
where $\gamma\in (0, 1)$, $\beta=(\beta_{1}, \ldots, \beta_{N})\in \N^{N}$, $|\beta|=\beta_{1}+\ldots+\beta_{N}$,
$$
\|u\|_{C^{2}(\overline{\Omega})}=\sum_{0\leq |\beta|\leq 2}\|D^{\beta}u\|_{C(\overline{\Omega})} \ \mbox{and} \ [D^{\beta}u]_{\gamma}=\sup_{x, y\in \Omega\atop x\neq y}\frac{|D^{\beta}u(x)-D^{\beta}u(y)|}{|x-y|^{\gamma}}.
$$
Moreover $Y$ will denote the Banach space $C^{0, \gamma}(\overline{\Omega})$ with norm
$$
\|f\|_{Y}=\|f\|_{C(\overline{\Omega})}+[f]_{\gamma},
$$
where $\|f\|_{C(\overline{\Omega})}=\max_{x\in \overline{\Omega}}|f(x)|$.

\medskip

Hereafter same symbol $C$ denotes different positive constants.

\medskip

\section{A nonlocal eigenvalue problem}\label{se:eigen}

In this section we are interested in studying the following nonlocal eigenvalue problem
\begin{equation}\label{E}\tag{$EP$}
\left \{ \begin{array}{ll}
-div\left(\displaystyle\frac{\nabla u}{c+|\nabla u|_{2}^{2}}\right)=\lambda \left\{ -div\left[\displaystyle\frac{\nabla c}{(c+|\nabla u|_{2}^{2})^{2}}\right]\right\} u & \mbox{in $\Omega$,}\\
u=0 & \mbox{on $\partial\Omega$,}
\end{array}\right.
\end{equation}
where $\Omega\subset\R^{N}$ is bounded smooth domain, $\lambda$ is a positive parameter and $c\in C^{2}(\overline{\Omega})$ is a positive (not constant) function. As we will see in the next section, problem \eqref{E} arises naturally when one studies questions of uniqueness to the problem \eqref{P}.

Before to state the main results of this section, we observe that 

\begin{lemma}\label{positive}
The set
$$
\mathcal{A}:=\left\{\alpha>0: -div\left[\frac{\nabla c}{(c+\alpha)^{2}}\right]>0 \ \mbox{in some open $\Omega_{0}\subset \Omega$}\right\}
$$
is not empty if, and only if, there is an open $\hat{\Omega}\subset\R^{N}$ such that
\begin{equation}\label{C}
\Delta c<2\frac{|\nabla c|^{2}}{c} \ \mbox{in} \ \hat{\Omega}.
\end{equation}
\end{lemma}

\begin{proof}
Differentiating we get
\begin{equation}\label{ineq}
-div\left[\frac{\nabla c}{(c+\alpha)^{2}}\right]=-\frac{1}{(c+\alpha)^{2}}\Delta c+\frac{2}{(c+\alpha)^{3}}|\nabla c|^{2}.
\end{equation}
Now, note that
\begin{equation}\label{div}
-div\left[\frac{\nabla c}{(c+\alpha)^{2}}\right]> 0 \ \mbox{in some open} \ \Omega_{0}
\end{equation}
if, and only if,
\begin{equation}\label{ineq1}
\Delta c< 2\frac{|\nabla c|^{2}}{(c+\alpha)} \ \mbox{in} \ \Omega_{0}.
\end{equation}
It is clear that the existence of a positive  number $\alpha$ satisfying \eqref{ineq1} is equivalent to inequality in \eqref{C}.
\end{proof}

\begin{remark}\label{ok}
In previous Lemma we have shown also that $\mathcal{A}=\emptyset$ if, and only if,  
\begin{equation}\label{ineq4}
\Delta c\geq 2\frac{|\nabla c|^{2}}{c} \ \mbox{in} \ \Omega.
\end{equation}
Certainly, there are many positive functions $c\in C^{2}(\overline{\Omega})$ verifying \eqref{ineq4}. For instance, setting $c=\delta e+1$, where $0<\delta\leq \min\{1/(4|\nabla e|_{\infty}^{2}), 1/(2|e|_{\infty})\}$ and 
\begin{equation}
\left \{ \begin{array}{ll}
\Delta e=1 & \mbox{in $\Omega$,}\\
e=0 & \mbox{on $\partial\Omega$,}
\end{array}\right.
\end{equation}
we conclude that $c>0$ and satisfies \eqref{ineq4}.
\end{remark}

\begin{remark}
An interesting question when \eqref{C} holds is about the topology of set $\mathcal{A}$. In this direction, the proof of Lemma \ref{positive} allows us to say that $\mathcal{A}$ contains ever a neighborhood $(0, \alpha_{0})$.
\end{remark}

Now we are ready to claim the following result.

\begin{theorem}\label{TE}
Suppose that \eqref{C} holds. For each $\alpha\in \mathcal{A}$, problem \eqref{E} has a unique solution $(\lambda_{\alpha}, u_{\alpha})$ such that $\lambda_{\alpha}>0$, $u_{\alpha}>0$ and $|\nabla u_{\alpha}|_{2}^{2}=\alpha$.
\end{theorem}

\begin{proof}
From Lemma \ref{positive}, $\mathcal{A}\neq \emptyset$. Since $c\in C^{2}(\overline{\Omega})$ and $b>0$ in $\overline{\Omega}$, it follows from Proposition \ref{Prodi} that, for each $\alpha\in \mathcal{A}$, the eigenvalue problem
\begin{equation}\label{EP}\tag{$P_{\alpha}$}
\left \{ \begin{array}{ll}
-div\left(\displaystyle\frac{\nabla u}{c+\alpha}\right)=\lambda \left\{ -div\left[\displaystyle\frac{\nabla c}{(c+\alpha)^{2}}\right]\right\} u & \mbox{in $\Omega$,}\\
u=0 & \mbox{on $\partial\Omega$,}
\end{array}\right.
\end{equation}
has a positive smallest eigenvalue $\lambda_{\alpha}$ whose associated eigenspace $V_{\alpha}$ is unidimensional and its eigenfunctions have defined sign. 
Choosing $u\in V_{\alpha}$ such that $u> 0$ and $|\nabla u|_{2}^{2}=\alpha$, the result follows.
\end{proof}

\begin{remark}\label{re}
In particular, if \eqref{C} holds then
\begin{equation}\label{PI}
\displaystyle\int_{\Omega}u_{\alpha}^{2}\left\{-div\left[\displaystyle\frac{\nabla c}{(c+\alpha)^{2}}\right]\right\} dx=\frac{1}{\lambda_{\alpha}}\displaystyle\int_{\Omega}\displaystyle\frac{|\nabla u_{\alpha}|^{2}}{c+\alpha} dx, \ \forall \ \alpha\in \mathcal{A}.
\end{equation}
\end{remark}

\begin{corollary}\label{EE}
Suppose \eqref{C}. For each $\alpha\in \mathcal{A}$, the following inequality holds
$$
\lambda_{\alpha}\geq \displaystyle\frac{\sqrt{\lambda_{1}}(c_{L}+\alpha)^{2}}{2|\nabla c|_{\infty}(c_{M}+\alpha)\displaystyle}, 
$$
where $\lambda_{1}$ is the first eigenvalue of laplacian operator with Dirichlet boundary condition, $c_{L}=\min_{x\in\overline{\Omega}}c(x)$, $c_{M}=\max_{x\in\overline{\Omega}}c(x)$ and $|\nabla c|_{\infty}=\max_{x\in\overline{\Omega}}|\nabla c (x)|$.
\end{corollary}

\begin{proof}
From Remark \ref{re}, we get
\begin{equation}\label{16}
\lambda_{\alpha}=\displaystyle\frac{\displaystyle\int_{\Omega}\displaystyle\frac{|\nabla u_{\alpha}|^{2}}{c+\alpha} dx}{\displaystyle\int_{\Omega}u_{\alpha}^{2}\left\{-div\left[\displaystyle\frac{\nabla c}{(c+\alpha)^{2}}\right]\right\} dx}.
\end{equation}
Observe that
\begin{equation}\label{17}
\displaystyle\int_{\Omega}\displaystyle\frac{|\nabla u_{\alpha}|^{2}}{c+\alpha} dx\geq \displaystyle\frac{\alpha}{c_{M}+\alpha}.
\end{equation}
Moreover, by using the Divergence Theorem,
$$
\displaystyle\int_{\Omega}u_{\alpha}^{2}\left\{-div\left[\displaystyle\frac{\nabla c}{(c+\alpha)^{2}}\right]\right\} dx=2\displaystyle\int_{\Omega}\frac{u_{\alpha}\nabla u_{\alpha}\nabla c}{(c+\alpha)^{2}} dx
\leq \displaystyle\frac{2|\nabla c|_{\infty}\displaystyle\int_{\Omega}u_{\alpha}|\nabla u_{\alpha}| dx}{(c_{L}+\alpha)^{2}}.
$$
From H\"older and Poincar\'e inequalities, we conclude that
\begin{equation}\label{18}
\displaystyle\int_{\Omega}u_{\alpha}^{2}\left\{-div\left[\displaystyle\frac{\nabla c}{(c+\alpha)^{2}}\right]\right\} dx\leq \displaystyle\frac{2|\nabla c|_{\infty}\alpha}{\sqrt{\lambda_{1}}(c_{L}+\alpha)^{2}}.
\end{equation}
From \eqref{16}, \eqref{17} and \eqref{18} we have
$$
\lambda_{\alpha}\geq \displaystyle\frac{\sqrt{\lambda_{1}}(c_{L}+\alpha)^{2}}{2|\nabla c|_{\infty}(c_{M}+\alpha)\displaystyle},
$$
for all $\alpha\in \mathcal{A}$.
\end{proof}

\section{Uniqueness results}\label{se:trivial}

In order to apply Theorem \ref{t2} we define operator $\Psi:X\to Y$ by
$$
\Psi(u)=\left( a(x)+b(x)\int_{\Omega}|\nabla u|^{2}dx\right)\Delta u.
$$

In the sequel, we will denote $M\left(x,|\nabla u|_{2}^{2}\right)=a(x)+b(x)\int_{\Omega}|\nabla u|^{2}dx$ for short, where $|\nabla u|_{2}^{2}=\int_{\Omega}|\nabla u|^{2}dx$. The proof of main results of this paper will be divided in various propositions.

\medskip

\begin{proposition}\label{proper}
Operator $\Psi:X\to Y$ is proper.
\end{proposition}

\begin{proof} 
It is sufficient to prove that if $\{h_n\}\subset Y$ is a sequence converging to $h \in Y$ and $\{u_{n}\}\subset X$ is another sequence with $\Psi(u_{n})=-h_{n}$ then $\{u_{n}\}$ has a convergent subsequence in $X$. For this, note that the equality $\Psi(u_{n})=-h_{n}$ is equivalent to
\begin{equation}\label{1}
-\Delta u_{n} = \frac{h_{n}}{M\left(x,|\nabla u_{n}|_{2}^{2}\right)}.
\end{equation}
Observe that $h_{n}/M\left(.,|\nabla u_{n}|_{2}^{2}\right)\in Y$ because $h_{n}\in Y$, $M\left(.,|\nabla u_{n}|_{2}^{2}\right)\in Y$ and $M\left(x,|\nabla u_{n}|_{2}^{2}\right)\geq a_{0}$.

Moreover,
\begin{equation}\label{2}
\left\|\frac{h_n(x)}{M\left(x,|\nabla u_{n}|_{2}^{2}\right)}\right\|_{C(\overline{\Omega})}\leq \|h_n\|_{C(\overline{\Omega})}/a_{0}, \ \forall \ n\in\N.
\end{equation}

\medskip

From $\|h_n\|_{C(\overline{\Omega})}\leq \|h_{n}\|_{Y}$, \eqref{2} and from boundedness of $\{h_{n}\}$ in $Y$ follows that $\left\{h_{n}/M\left(x,|\nabla u_{n}|_{2}^{2}\right)\right\}$ is bounded in $C(\overline{\Omega})$. Thus, the continuous embedding from $C^{1,\gamma}(\overline{\Omega})$ into $C(\overline{\Omega})$ and equality in (\ref{1}) tell us that $\{u_{n}\}$ is bounded in $C^{1,\gamma}(\overline{\Omega})$ (see Theorem 0.5 in \cite{AP}). Finally, by compact embedding from $C^{1,\gamma}(\overline{\Omega})$ into $C^{1}(\overline{\Omega})$, we conclude that there exists $u\in C^{1}(\overline{\Omega})$ such that, passing to a subsequence, 
\begin{equation}
u_{n}\to u \ \mbox{in} \ C^{1}(\overline{\Omega}).
\end{equation}
Last convergence leads to
\begin{equation}
|\nabla u_{n}(x)|^{2}\to |\nabla u(x)|^{2} \ \mbox{uniformly in $x\in \Omega$}.
\end{equation}
Whence
\begin{equation}\label{3}
|\nabla u_{n}|_{2}^{2}\to |\nabla u|_{2}^{2}.
\end{equation}

\medskip

In the follows, we show that 
\begin{equation}\label{4}
  \left\|\frac{h_n}{M\left(., |\nabla u_{n}|_{2}^{2}\right)}\right\|_{Y}\leq C,
\end{equation}
for some positive constant $C$. In fact, since $\{h_{n}\}\subset Y$ and $\left\{M\left(.,|\nabla u_{n}|_{2}^{2}\right)\right\}\subset Y$, with $M(x, t)\geq a_{0}>0$ for all $t\geq 0$, a straightforward calculation shows us that
$$
\left[ \frac{h_n}{M\left(.,|\nabla u_{n}|_{2}^{2}\right)}\right]_{\gamma}\leq  \frac{1}{a_{0}^2} 
  \left( \|h_{n}\|_{C(\overline{\Omega})} \left[M\left(.,|\nabla u_{n}|_{2}^{2}\right)\right]_{\gamma}+ \left\|M\left(., |\nabla u_{n}|_{2}^{2}\right)\right\|_{C(\overline{\Omega})} [h_n]_{\gamma}\right).
$$

From $\|h_n\|_{C(\overline{\Omega})}, [h_n]_{\gamma} \leq C$,
\begin{equation}
  \left[M\left(.,|\nabla u_{n}|_{2}^{2}\right)\right]_{\gamma}\leq [a]_{\gamma}+[b]_{\gamma}|\nabla u_{n}|_{2}^{2}  \leq [a]_\gamma + C[b]_\gamma
\end{equation}
and
\begin{equation}
  \left\|M\left(.,|\nabla u_{n}|_{2}^{2}\right)\right\|_{C(\overline{\Omega})}\leq \|a\|_{C(\overline{\Omega})}+\|b\|_{C(\overline{\Omega})}|\nabla u_{n}|_{2}^{2}  \leq \|a\|_{C(\overline{\Omega})} + C\|b\|_{C(\overline{\Omega})}
\end{equation}
it follows that
\begin{equation*}
 \left[ \frac{h_n}{M\left(.,|\nabla u_{n}|_{2}^{2}\right)}\right]_{\gamma}
  \leq \frac{C}{a_0^2} \left( [a]_\gamma + C [b]_\gamma + \|a\|_{C(\overline{\Omega})} + C \|b\|_{C(\overline{\Omega})}\right)
  = \frac{C}{a_0^2} \|a\|_{Y} + \frac{C^2}{a_0^2} \|b\|_{Y}.
\end{equation*}
Being $\left\{h_{n}/M\left(x,|\nabla u_{n}|_{2}^{2}\right)\right\}$ bounded in $C(\overline{\Omega})$, the last inequality proves the assertion in \eqref{4}.

By \eqref{1}, \eqref{4} and Theorem 0.5 in \cite{AP}, sequence $\{u_{n}\}$ is bounded in $X$. By compact embedding from $X$ in $C^{2}(\overline{\Omega})$, passing to a subsequence, we get
\begin{equation}\label{5}
  u_n \to u \ \mbox{in} \ C^2(\overline{\Omega}).
\end{equation}

By \eqref{5}, passing to the limit in $n\to \infty$ in \eqref{1} we have
\begin{equation}
-\Delta u=\frac{h}{M\left(x,|\nabla u|_{2}^{2}\right)}.
\end{equation}
Last equality and Theorem 0.5 in \cite{AP} allow us to conclude that $u\in X$.

Finally, by linearity of laplacian, we have
\begin{equation}\label{6}
  -\Delta(u_n - u) =
 \frac{h_n}{M\left(x,|\nabla u_{n}|_{2}^{2}\right)}- \frac{h}{M\left(x,|\nabla u|_{2}^{2}\right)}.
\end{equation}
From \eqref{6} and Theorem 0.5 in \cite{AP} we conclude that $u_{n}\to u$ in $X$.
\end{proof}

\medskip

\begin{proposition}\label{Prop}
Let $a, b\in C^{0, \gamma}(\overline{\Omega})$ and $u\in X$.  If
\begin{equation}\label{Int}
\int_{\Omega}\frac{b(x)u\Delta u}{M(x,|\nabla u|_{2}^{2})}dx\neq 1/2
\end{equation}
holds then $\Psi$ is locally invertible in $u$.
\end{proposition}

\begin{proof} 
We are interested in using Theorem \ref{t1} to prove this Lemma. It is standard to show that $\Psi\in C^{1}(X, Y)$ and
$$
\Psi'(u)v=2b(x)\Delta u\int_{\Omega}\nabla u\nabla vdx+M(x,|\nabla u|_{2}^{2})\Delta v.
$$
Remain us proving that $\Psi'(u): X\to Y$ is an isomorphism. It is clear that if $u=0$ there is nothing to prove. Now, if $u\neq 0$, observes that $\Psi'(u)$ is an isomorphism if, and only if, for each $g\in Y$ given, there is a unique $v\in X$ such that $\Psi'(u)v=-g$, this is
\begin{equation}\label{7}
-M(x,|\nabla u|_{2}^{2})\Delta v=g(x)+2b(x)\Delta u\int_{\Omega}\nabla u\nabla vdx.
\end{equation}
From Divergence Theorem, equation in (\ref{7}) is equivalent to
\begin{equation}\label{8}
-M(x,|\nabla u|_{2}^{2})\Delta v=g(x)-2b(x)\Delta u\int_{\Omega}u\Delta vdx.
\end{equation}
Consequently, $\Psi'(u)$ is an isomorphism if, and only if, for each $g\in Y$ given, there is a unique $v\in X$ such that
\begin{equation}\label{9}
\Delta v=\frac{2b(x)\Delta u\int_{\Omega}u\Delta vdx}{M(x,|\nabla u|_{2}^{2})}-\frac{g(x)}{M(x,|\nabla u|_{2}^{2})}.
\end{equation}

\medskip

To study equation (\ref{9}) we define the mapping $T: Y\to Y$ by
\begin{equation}\label{10}
T(w)=\frac{2b(x)\Delta u\int_{\Omega}u w dx}{M(x,|\nabla u|_{2}^{2})}-\frac{g(x)}{M(x,|\nabla u|_{2}^{2})}
\end{equation}
and we note that, since for each $w\in Y$ problem
\begin{equation}\label{LP}\tag{LP}
\left \{ \begin{array}{ll}
\Delta z=w(x) & \mbox{in $\Omega$,}\\
z=0 & \mbox{on $\partial\Omega$,}
\end{array}\right.
\end{equation}
has a unique solution $z\in X$, looking for solutions of (\ref{9}) is equivalent to find fixed points of $T$. Denoting $t=\int_{\Omega}u w dx$, it follows that $w$ is a fixed point of $T$ if, and only if,
\begin{equation}\label{11}
w=T(w)=t\frac{2b(x)\Delta u}{M(x,|\nabla u|_{2}^{2})}-\frac{g(x)}{M(x,|\nabla u|_{2}^{2})}.
\end{equation}

\medskip

%
%

\medskip

Therefore $w$ is a fixed point of $T$ if, and only if,
$$
T\left( t\frac{2b(x)\Delta u}{M(x,|\nabla u|_{2}^{2})}-\frac{g(x)}{M(x,|\nabla u|_{2}^{2})}\right)=t\frac{2b(x)\Delta u}{M(x,|\nabla u|_{2}^{2})}-\frac{g(x)}{M(x,|\nabla u|_{2}^{2})}.
$$
From (\ref{10}), we get
$$
\frac{2b(x)\Delta u}{M(x,|\nabla u|_{2}^{2})}\int_{\Omega}u\left[ t\frac{2b(x)\Delta u}{M(x,|\nabla u|_{2}^{2})}-\frac{g(x)}{M(x,|\nabla u|_{2}^{2})}\right]dx=t\frac{2b(x)\Delta u}{M(x,|\nabla u|_{2}^{2})}.
$$
Since $b>0$ and $\Delta u\not\equiv 0$ (because $u\neq 0$), $T$ admits a fixed point if, and only if,
$$
2\int_{\Omega}u\left[ t\frac{2b(x)\Delta u}{M(x,|\nabla u|_{2}^{2})}-\frac{g(x)}{M(x,|\nabla u|_{2}^{2})}\right]dx=t,
$$
namely,
\begin{equation}\label{13}
t\left[ \int_{\Omega}\frac{2b(x)u\Delta u}{M(x,|\nabla u|_{2}^{2})}dx-1\right]=2\int_{\Omega}\frac{g(x)u}{M(x, |\nabla u|_{2}^{2})}dx.
\end{equation}

\medskip

Equality (\ref{13}) say us that if \eqref{Int} occurs then $T$ has a unique fixed point $w$ given by
$$
w=t\frac{2b(x)\Delta u}{M(x,|\nabla u|_{2}^{2})}-\frac{g(x)}{M(x,|\nabla u|_{2}^{2})},
$$
with
$$
t=2\int_{\Omega}\frac{g(x)u}{M(x, |\nabla u|_{2}^{2})}dx/\left[ \int_{\Omega}\frac{2b(x)u\Delta u}{M(x,|\nabla u|_{2}^{2})}dx-1\right].
$$
\end{proof}

\medskip

\begin{remark}
Equality \eqref{13} shows us that $\Psi'(u):X\to Y$ is not surjective if 
$$
\int_{\Omega}\frac{b(x)u\Delta u}{M(x,|\nabla u|_{2}^{2})}dx= 1/2.
$$
In fact, in this case, functions $g\in Y$ such that
$$
\int_{\Omega}\frac{g(x)u}{M(x, |\nabla u|_{2}^{2})}dx\neq 0
$$
are not in the range of $\Psi'(u)$.
\end{remark}

Actually, it is possible to get the same result of (existence and) uniqueness provided in \cite{FMSS} for signed functions as a consequence of Global Inverse Theorem and previous Proposition. This is exactly the content of next corollary.

\begin{corollary}\label{alt}
For each signed function $h\in Y$ given, problem (\ref{P}) has a unique solution.
\end{corollary}

\begin{proof} 
First of all, we define the sets 
$$
P_{1}=\{u\in X: \Delta u\geq 0\}\subset X
$$ 
and 
$$
P_{2}=\{h\in Y: h\geq 0\}\subset Y.
$$ 
Consider $P_{1}\cup (-P_{1})$ and $P_{2}\cup (-P_{2})$ as metric spaces whose metrics are induced from $X$ and $Y$, respectively. 

It is clear that $P_{1}\cup (-P_{1})$ is arcwise connected (because $P_{1}$ and $-P_{1}$ are convex sets and $P_{1}\cap(-P_{1})=\{0\}$) closed in $X$. On the other hand, since $P_{2}\cup (-P_{2})$ is the union of the closed cone of nonnegative functions of $Y$ with the closed cone of nonpositive functions of $Y$, follows that $P_{2}\cup (-P_{2})$ is simply connected. 

From $\Psi(P_{1})\subset P_{2}$ and $\Psi(-P_{1})\subset (-P_{2})$, it follows that $\Psi$ is well defined from $P_{1}\cup (-P_{1})$ to $P_{2}\cup (-P_{2})$.

Moreover, being $\Psi$ proper from $X$ to $Y$ (see Proposition \ref{proper}) and $P_{1}\cup (-P_{1})$ and $P_{2}\cup (-P_{2})$ are closed metric spaces in $X$ and $Y$, respectively, it follows that $\Psi$ is proper from $P_{1}\cup (-P_{1})$ to $P_{2}\cup (-P_{2})$.

Note that if $u\in P_{1}$ (resp. $-P_{1}$) then, as $u$ is (the unique) solution to problem
\begin{equation}
\left \{ \begin{array}{ll}
\Delta u=\Delta u & \mbox{in $\Omega$,}\\
u=0 & \mbox{on $\partial\Omega$.}
\end{array}\right.
\end{equation}
Follows from maximum principle that $u\leq 0$ (resp. $u \geq 0$). Whence, we have
$$
\int_{\Omega}\frac{b(x)u\Delta u}{M(x,|\nabla u|_{2}^{2})}dx\leq 0, \ \forall \ u\in P_{1}\cup (-P_{1}).
$$
Therefore, from Proposition \ref{Prop}, $\Psi:P_{1}\cup (-P_{1})\to P_{2}\cup (-P_{2})$ is locally invertible. The result follows now from Global Inverse Theorem. 
\end{proof}

Next corollary does not ensure uniqueness of solution for problem \eqref{P} when function $h$ given is sign changing, but it tells us that there is a unique solution with ``little variation'' if $h\in Y$ given (signed or not) has ``little variation''.

\begin{corollary}
There are  positive constants $\varepsilon,\delta$ such that for each $h\in Y$ with $\|h\|_{Y}<\varepsilon$, problem \eqref{P} has a unique solution $u$ with $\|u\|_{X}<\delta$.
\end{corollary}

\begin{proof} 
It is sufficient to note that when $u=0$ the integral in previous proposition is null.
\end{proof}

We are now ready to prove Theorem \ref{Main1}.

\medskip

{\sl Proof  of Theorem \ref{Main1}.}

\medskip
Since $X$ and $Y$ are Banach spaces then $X$ is arcwise connected and $Y$ is simply connected. Moreover, from Proposition \ref{proper}, operator $\Psi$ is proper and  by Divergence Theorem
$$
\int_{\Omega}\frac{b(x)u\Delta u}{M(x,|\nabla u|_{2}^{2})}dx=\frac{1}{\theta+|\nabla u|_{2}^{2}}\int_{\Omega}u\Delta u dx=-\frac{|\nabla u|_{2}^{2}}{\theta+|\nabla u|_{2}^{2}}<0, \ \forall \ u\in X.
$$
The result follows directly from Proposition \ref{Prop} and Global Inverse Theorem.
$\square$

\medskip

Next proposition provides us a sufficient condition on functions $a$ and $b$ for that \eqref{Int} occurs when $a/b$ is not constant.

\begin{proposition}\label{Eureka}
Let $a, b\in C^{2, \gamma}(\overline{\Omega})$ and $c=a/b$. 
\begin{itemize}
\item[$(i)$] If $\Delta c\geq 2|\nabla c|^{2}/c \ \mbox{in} \ \Omega$, then  
\begin{equation}
\int_{\Omega}\frac{b(x)u\Delta u}{M(x,|\nabla u|_{2}^{2})}dx\leq 0, \ \forall \ u\in X.
\end{equation}

\item[$(ii)$] If $\Delta c< 2|\nabla c|^{2}/c \ \mbox{in some open} \ \Omega_{0}\subset \Omega$ then
\begin{equation}\label{20}
\int_{\Omega}\frac{b(x)u\Delta u}{M(x,|\nabla u|_{2}^{2})}dx<1/2, \ \forall \ u\in X,
\end{equation}
provided that
$$
\frac{|\nabla c|_{\infty}c_{M}}{\sqrt{\lambda_{1}}c_{L}^{2}}\leq 3/2.
$$
\end{itemize}
\end{proposition}

\begin{proof} 
Putting $b$ in evidence in the integral (\ref{Int}), we get
$$
\int_{\Omega}\frac{b(x)u\Delta u}{M(x,|\nabla u|_{2}^{2})}dx=\int_{\Omega}\frac{u\Delta u}{c+|\nabla u|_{2}^{2}}dx,
$$
where $c=c(x)=a(x)/b(x)$. From Divergence Theorem, we have
$$
\int_{\Omega}\frac{u\Delta u}{c+|\nabla u|_{2}^{2}}dx=-\int_{\Omega}\nabla\left(\frac{u}{c+|\nabla u|_{2}^{2}}\right)\nabla u dx.
$$
Since 
$$
\nabla\left(\frac{u}{c+|\nabla u|_{2}^{2}}\right)=\frac{1}{c+|\nabla u|_{2}^{2}}\nabla u- \frac{u}{(c+|\nabla u|_{2}^{2})^{2}}\nabla c,
$$
we conclude that
\begin{eqnarray*}
\int_{\Omega}\frac{b(x)u\Delta u}{M(x,|\nabla u|_{2}^{2})}dx&=&-\int_{\Omega}\frac{|\nabla u|^{2}}{c+|\nabla u|_{2}^{2}}dx+
\int_{\Omega}\frac{u\nabla u\nabla c}{(c+|\nabla u|_{2}^{2})^{2}}dx\\
&=& -\int_{\Omega}\frac{|\nabla u|^{2}}{c+|\nabla u|_{2}^{2}}dx+\frac{1}{2}
\int_{\Omega}\frac{\nabla(u^{2})\nabla c}{(c+|\nabla u|_{2}^{2})^{2}}dx.
\end{eqnarray*}
Using again the Divergence Theorem
\begin{equation}\label{14}
\int_{\Omega}\frac{b(x)u\Delta u}{M(x,|\nabla u|_{2}^{2})}dx=-\int_{\Omega}\frac{|\nabla u|^{2}}{c+|\nabla u|_{2}^{2}}dx+\frac{1}{2}
\int_{\Omega}u^{2}\left\{-div\left[\frac{\nabla c}{(c+|\nabla u|_{2}^{2})^{2}}\right]\right\}dx.
\end{equation}

(i) At this case, from Lemma \ref{positive} (see also Remark \ref{ok}), $\mathcal{A}=\emptyset$ and, consequently, for each $u\in X$ we have
$$
\int_{\Omega}u^{2}\left\{-div\left[\frac{\nabla c}{(c+|\nabla u|_{2}^{2})^{2}}\right]\right\}dx\leq 0.
$$
Whence, by \eqref{14},
$$
\int_{\Omega}\frac{b(x)u\Delta u}{M(x,|\nabla u|_{2}^{2})}dx\leq 0, \ \forall \ u\in X.
$$

(ii) In this case $\mathcal{A}\neq \emptyset$. If $u\in X$ is such that $|\nabla u|_{2}^{2}\not\in\mathcal{A}$ we saw already that
$$
\int_{\Omega}\frac{b(x)u\Delta u}{M(x,|\nabla u|_{2}^{2})}dx\leq 0.
$$
Now, if $u\in X$ is such that $|\nabla u|_{2}^{2}\in\mathcal{A}$ then, from \eqref{14} and Proposition \ref{TE} (see also Remark \ref{re}), we obtain
\begin{equation}\label{15}
\int_{\Omega}\frac{b(x)u\Delta u}{M(x,|\nabla u|_{2}^{2})}dx\leq \left(\frac{1}{2\lambda_{\alpha}}-1\right)\int_{\Omega}\frac{|\nabla u|^{2}}{c+\alpha} dx,
\end{equation}
where $\alpha:=|\nabla u|_{2}^{2}$. If $\alpha$ is such that $1/2\leq \lambda_{\alpha}$ then, by \eqref{15}, $\int_{\Omega}b(x)u\Delta u/M(x,|\nabla u|_{2}^{2}) dx\leq 0$. Finally, if $0<\lambda_{\alpha}<1/2$ follows from $\alpha=|\nabla u|_{2}^{2}$ and from Corollary \ref{EE} that,
\begin{equation}\label{19}
\int_{\Omega}\frac{b(x)u\Delta u}{M(x,|\nabla u|_{2}^{2})}dx<\frac{|\nabla c|_{\infty}(c_{M}+\alpha)}{\sqrt{\lambda_{1}}(c_{L}+\alpha)^{2}}-1=:g(\alpha).
\end{equation}
We have that $g(0)=|\nabla c|_{\infty}c_{M}/\sqrt{\lambda_{1}}c_{L}^{2}-1$ and 
$$
g'(\alpha)=\frac{|\nabla c|_{\infty}(c_{L}-2c_{M}-\alpha)}{\sqrt{\lambda_{1}}(c_{L}+\alpha)^{3}}<0, \ \forall \alpha>0.
$$ 
Therefore $g$ is decreasing and, from \eqref{19}, we conclude that if 
$$
\frac{|\nabla c|_{\infty}c_{M}}{\sqrt{\lambda_{1}}c_{L}^{2}}\leq\frac{3}{2}
$$
then \eqref{20} holds.
\end{proof}

Bellow we give the proof of our main uniqueness result to problem \eqref{P} which covers sign changing functions.

{\sl Proof of Theorem \ref{Main2}.}
It follows directly from Proposition \ref{proper}, Proposition \ref{Prop}, Proposition \ref{Eureka} and Global Inverse Theorem.
$\square$

Theorems \ref{Main1} and \ref{Main2} seem to indicate that in the case that $h$ is sign changing the uniqueness of solution to the problem \eqref{P} is, in some way, related with the variation of $a/b$. In any way, remains open the question to know what happens with the number of solutions of \eqref{P} in the case that $h$ is sign changing, $\Delta c< 2|\nabla c|^{2}/c \ \mbox{in some open} \ \Omega_{0}\subset \Omega$ and $|\nabla c|_{\infty}c_{M}/\sqrt{\lambda_{1}}c_{L}^{2}$ is large.

\end{document}